\documentclass[11pt]{article}

\usepackage{amsmath, amssymb, amsfonts, ifthen, verbatim, fullpage}

\newtheorem{theorem}{Theorem}[section]
\newtheorem{corollary}[theorem]{\bf Corollary}
\newtheorem{lemma}[theorem]{\bf Lemma}
\newtheorem{proposition}[theorem]{\bf Proposition}
\newtheorem{conjecture}[theorem]{\bf Conjecture}
\newtheorem{problem}[theorem]{\bf Problem}

\newcommand{\DEF}[1]{{\em #1\/}}

\newcommand{\qed}{\hfill$\square$\bigskip}
\newenvironment{proof}[1][]%
{\ifthenelse{\equal{#1}{}}{\noindent\textit{Proof.
    }}{\noindent\textit{#1. }}}%
{\qed}

\begin{document}

\title{Minimum degree condition forcing complete graph immersion}

\author{
   Matt DeVos\thanks{Supported in part by an NSERC Discovery Grant (Canada) and a Sloan Fellowship.}\\[1mm]
   {Department of Mathematics}\\{Simon Fraser University}\\{Burnaby, B.C. V5A 1S6}
\and
   Zden\v{e}k Dvo\v{r}\'ak\thanks{Supported in part by the grant GA201/09/0197 of Czech Science Foundation
   and by Institute for Theoretical Computer Science (ITI), project 1M0021620808 of Ministry of~Education of~Czech Republic.}\\[1mm]
   Department of Applied Mathematics\\Charles University\\Prague, Czech Republic
\and
   Jacob Fox\thanks{Supported in part by a Simons Fellowship.}\\[1mm]
   Department of Mathematics\\MIT\\Cambridge, MA 02139
\and
   Jessica McDonald\thanks{Supported by an NSERC Postdoctoral fellowship.}\\[1mm]
   {Department of Mathematics}\\{Simon Fraser University}\\{Burnaby, B.C. V5A 1S6}
\and
   Bojan Mohar\thanks{Supported in part by an NSERC Discovery Grant (Canada), by the Canada Research Chair program, and by the Research Grant P1--0297 of ARRS (Slovenia).}~~\thanks{On leave from: IMFM \& FMF, Department of Mathematics, University of Ljubljana, Ljubljana, Slovenia.}\\[1mm]
   {Department of Mathematics}\\{Simon Fraser University}\\{Burnaby, B.C. V5A 1S6}
\and
   Diego Scheide\thanks{Postdoctoral fellowship at Simon Fraser University, Burnaby.}\\[1mm]
   {Department of Mathematics}\\{Simon Fraser University}\\{Burnaby, B.C. V5A 1S6}
}


\maketitle

\begin{abstract}
An immersion of a graph $H$ into a graph $G$ is a one-to-one mapping
$f:V(H) \to V(G)$ and a collection of edge-disjoint paths in $G$, one for
each edge of $H$, such that the path $P_{uv}$ corresponding to edge $uv$
has endpoints $f(u)$ and $f(v)$. The immersion is strong if the paths $P_{uv}$
are internally disjoint from $f(V(H))$.
It is proved that for every positive integer $t$, every simple graph of minimum
degree at least $200t$ contains a strong immersion of the complete graph $K_t$.
For dense graphs one can say even more. If the graph has order $n$ and has $2cn^2$
edges, then there is a strong immersion of the complete graph on at least $c^2 n$
vertices in $G$ in which each path $P_{uv}$ is of length 2.
As an application of these results, we resolve a problem raised by Paul Seymour
by proving that the line graph of every simple graph with average degree $d$ has
a clique minor of order at least $cd^{3/2}$, where $c>0$ is an absolute constant.

For small values of $t$, $1\le t\le 7$, every simple graph of minimum degree
at least $t-1$ contains an immersion of $K_t$ (Lescure and Meyniel \cite{LM},
DeVos et al.\ \cite{DKMO}). We provide a general class of examples showing
that this does not hold when $t$ is large.
\end{abstract}


\section{Overview}

In this paper, all graphs are finite and may have loops and multiple edges.
A graph $H$ is a \DEF{minor} of a graph $G$ if (a graph isomorphic to) $H$
can be obtained from a subgraph of $G$ by contracting edges.
A graph $H$ is a \DEF{topological minor} of a graph $G$ if $G$
contains a subgraph which is isomorphic to a graph that can be obtained
from $H$ by subdividing some edges. In such a case, we also say that $G$
contains a \DEF{subdivision} of $H$.
These two ``containment'' relations are well studied, and the theory of
graph minors developed by Robertson and Seymour provides a deep understanding
of them.

In this paper, we consider a different containment relation -- graph
immersions. A pair of distinct adjacent edges $uv$ and $vw$
is \DEF{split off} from their common vertex $v$ by deleting the edges
$uv$ and $vw$, and adding the edge $uw$ (possibly in parallel to an existing edge,
and possibly forming a loop if $u = w$). A graph $H$ is
said to be \DEF{immersed} in a graph $G$ if a graph isomorphic to
$H$ can be obtained from a subgraph of $G$ by splitting off pairs of
edges (and removing isolated vertices).
If $H$ is immersed in a graph $G$, then we also say
that $G$ \DEF{contains} an \DEF{$H$-immersion}.
An alternative definition is that $H$ is immersed in $G$ if there is
a 1-1 function $\phi: V(H)\to V(G)$ such that for each edge $uv\in E(H)$,
there is a path $P_{uv}$ in $G$ joining vertices $\phi(u)$ and $\phi(v)$,
and the paths $P_{uv}$, $uv\in E(H)$, are pairwise edge-disjoint. In this language, it is natural to consider a stricter definition: if the paths $P_{uv}$, $uv\in E(H)$, are internally disjoint from $\phi(V(H))$,
then the immersion is said to be \DEF{strong}.
Clearly, subdivision containment implies minor and strong immersion containment. On the other hand, minor and immersion relations are incomparable.

Some previous investigations on immersions have been conducted from
an algorithmic viewpoint \cite{im1,im2}, while less attention has been paid to structural
issues. However, Robertson and Seymour have
extended their celebrated proof of Wagner's conjecture
\cite{RS20} to prove that graphs are well-quasi-ordered by the
immersion relation \cite{RS23}, thus confirming a conjecture of Nash-Williams.
The proof is based on a significant part of the graph minors project.
Hence, we may expect that a structural approach concerning immersions is
difficult, maybe as difficult as structure results concerning graph
minors.

An important motivation to study graph minors is given by
classical conjectures of Haj\'os and Hadwiger. They relate the chromatic
number $\chi(G)$ of a graph $G$ and (topological) clique minor containment in $G$.
Hadwiger's Conjecture \cite{hadwiger}, which dates back to 1943,
suggests a far-reaching generalization of the Four-Color Theorem and
is considered to be one of the deepest open problems in graph
theory. It states that every loopless graph without a $K_k$-minor is
$(k-1)$-col{ou}rable. The special case of the Hadwiger Conjecture when
$k=5$ is equivalent to the Four-Color Theorem. Robertson, Seymour and Thomas
\cite{RST1} proved the case when $k=6$. The cases with $k \geq 7$ are still open.
Haj\'os proposed a stronger conjecture in the 1940's that for all $k \geq 1$,
every $k$-chromatic graph contains a subdivision of the complete
graph on $k$ vertices. However, this conjecture was disproved for every
$k \geq 7$ by Catlin \cite{catlin}.

In analogy, it would be interesting to investigate relations between the chromatic number
of a graph $G$ and the largest size of a complete graph immersion.
Abu-Khzam and Langston conjectured the following in \cite{al}.

\begin{conjecture}
\label{conj1}
The complete graph $K_k$ can be immersed in any $k$-chromatic graph.
\end{conjecture}

This conjecture, like Hadwiger's conjecture and Haj\'os' conjecture,
is trivially true for $k \leq 4$. One of the outcomes of \cite{DKMO}
is that Conjecture \ref{conj1} holds for every $k\le 7$.
Although the results of this paper are not directly related to Conjecture \ref{conj1},
they provide some evidence that supports this conjecture.

For subdivisions and minor containment, there are classical results showing that large average degree in a simple graph (or, equivalently, large minimum degree) forces the complete graph $K_t$ as a (topological) minor.
Average degree $\Omega(t\sqrt{\log t})$ forces $K_t$ as a minor (Kostochka \cite{Ko}, Thomason \cite{Tho}), and this bound is best possible. For subdivisions, an old conjecture of Mader and Erd\H{o}s-Hajnal, which was eventually proved by Bollob\'as and Thomason \cite{BoTh} and independently by Koml\'os and Szemer\'edi \cite{KoSz}, says that there is a constant $c$ such that every graph with average degree at least $ct^2$ contains a subdivision of $K_t$, and this is tight apart from the constant $c$.
The corresponding extremal problem for complete graph immersions has been proposed in~\cite{DKMO}.

\begin{problem}
\label{prb2} Let $t$ be a positive integer.
Find the smallest value $f(t)$ such that every simple graph of minimum degree
at least $f(t)$ contains an immersion of~$K_{t}$.
\end{problem}

Clearly, $f(t)\ge t-1$. It has been proved by Lescure and Meyniel \cite{LM} and
DeVos et al.\ \cite{DKMO} that $f(t)=t-1$ for $t\le7$.
An example due to Paul Seymour (see \cite{DKMO} or \cite{LM}) shows
that $f(t)\ge t$ for every $t\ge 10$. Seymour's example is the graph obtained from the complete
graph $K_{12}$ by removing the edges of four disjoint triangles. This graph does not
contain $K_{10}$-immersion.
In Section \ref{sect:examples}, we provide a general class of examples with minimum degree
$t-1$ and without $K_t$-immersions.

The aforementioned results imply that $f(t) = O(t^2)$. However, subquadratic dependence does not follow easily from known results.
The main result of this paper is the following theorem which implies that $f(t)\le 200t$.

\begin{theorem}
\label{main}
Every simple graph with minimum degree at least\/ $200t$ contains a strong immersion of $K_t$.
\end{theorem}

The proof of Theorem \ref{main} is found in Section 3 of this paper, following some preliminary results in Section 2. It is worth noting that a linear upper bound for the dense case (when $G$ has at least $cn^2$ edges) can also be proved using Szemer\'edi's Regularity Lemma (private communication by Jan Hladk\'y). However, in the dense case, we obtain a stronger result as we will now discuss.

Recall that a subdivision of a graph $H$ is a graph formed by replacing edges of $H$ by internally vertex disjoint paths. In the special case in which each of the paths replacing the edges of $H$ has exactly $k$ internal vertices, it is called a \DEF{$k$-subdivision} of $H$. The aforementioned result of Bollob\'as and Thomason \cite{BoTh} and Koml\'os and Szemer\'edi \cite{KoSz} implies that any $n$-vertex graph with $cn^2$ edges contains a topological copy of a complete graph on at least $c'\sqrt{n}$ vertices. An old question of Erd\H{o}s \cite{Er77} asks whether one can strengthen this statement and find in every graph $G$ on $n$ vertices with $cn^2$ edges a $1$-subdivision of a complete graph with at least $c'\sqrt{n}$ vertices for some positive $c'$ depending on $c$. After several partial results, this question was settled affirmatively by Alon, Krivelevich, and Sudakov \cite{AKS}, who showed that one can take $c'=\Omega(c)$. As shown in \cite{FoSu}, this bound is tight in the sense that for every $\epsilon>0$, this bound cannot be improved to $\Omega(c^{1-\epsilon})$.

A \DEF{$k$-immersion} of a graph $H$ into a graph $G$ is an immersion of $H$
into $G$ such that each of the paths has exactly $k$ internal vertices. For
example, a $0$-immersion of $H$ into a graph $G$ is simply a copy of $H$ as a
subgraph in $G$. In Section \ref{sect:dense} we prove the following theorem
which shows that every dense graph contains a $1$-immersion of a clique of
linear size. It shows that, over all graphs of a given order and size, the
order of the largest guaranteed clique $1$-immersion is about the square of the
order of the largest guaranteed clique $1$-subdivision.

\begin{theorem}
\label{main1immerse}
Let\/ $G$ be a simple graph on $n$ vertices and $2cn^2$ edges. Then there is a strong\/ $1$-immersion of the complete graph on at least $c^2 n$ vertices in $G$.
\end{theorem}

The proof is deferred to Section \ref{sect:dense}, where we also show that the dependence between the constants in the theorem is essentially best possible.

Paul Seymour recently proposed several problems regarding clique
minors in line graphs. Note that a clique minor in a line graph is
equivalent to finding edge-disjoint connected subgraphs (each with at least one edge) of the original
graph, each pair intersecting in at least one vertex. The following is a quick corollary of Theorem \ref{main}.

\begin{corollary}
\label{cor}
The line graph of any simple graph with average degree $d$ has a clique minor of
order at least $c d^{3/2}$, where $c>0$ is an absolute constant.
\end{corollary}

The proof of Corollary \ref{cor} is given in Section \ref{sect:6}, where we also show that
this result is best possible up to the constant~$c$.

\section{Preliminary results}

In this section we outline some preliminary lemmas to be used in the proof of our
main theorem.

\begin{lemma}
\label{bip-im}
If\/ $a \le b$ then the complete bipartite graph $K_{a,b}$ contains a strong $K_a$-immersion.
\end{lemma}

\begin{proof}
Let the bipartition of $K_{a,b}$ be $(A,B)$ with $|A| = a$ and $|B| = b$ and consider a complete graph on the vertices of $A$ which is properly edge-col{ou}red using at most $a$ col{ou}rs.
Such a col{ou}ring is easy to construct. Now, associate each col{ou}r with a vertex in $B$ and
then split off pairs of edges incident with this vertex to add the matching on $A$ consisting of the edges of its associated col{ou}r.
\end{proof}

In what follows, we let $\bar{d}(G)$ denote the average degree of the graph $G$.
If $v\in V(G)$, then we denote by $N(v)$ the set of neighb{ou}rs of $v$ in $G$.

\begin{lemma}
\label{lem:split1}
Let\/ $G$ be a simple $n$-vertex graph with no subgraph isomorphic to $K_q$.
If $u \in V(G)$ has ${\mathit deg}(u) \le \bar{d}(G)$, then there exists a simple graph $G'$ obtained by splitting the vertex $u$ so that\/ $\bar{d}(G') \ge \bar{d}(G) - \frac{q-2}{n-1}$.
\end{lemma}

\begin{proof}
Let $H$ be the complement of the graph induced on $N(u)$ and choose a maximal
matching $M$ in $H$.  Let $X$ be the subset of $N(u)$ not covered by this matching.  Next, we form a simple graph
$G'$ by splitting off pairs of edges incident with the vertex $u$ to add the new edges $M$, and delete the edges from $u$ to vertices in $X$.  Now, $X$ is
an independent set in $H$, so it is a clique in $G$ joined completely to $u$.  Thus $|X| \le q-2$ and we find that
\begin{align*}
 \bar{d}(G') &= \frac{1}{n-1} \sum_{v \in V(G) \setminus \{u\}} {\mathit deg}_{G'}(v)\\
 &\ge \frac{1}{n-1} \sum_{v \in V(G) \setminus \{u\}} {\mathit deg}_{G}(v) - \frac{q-2}{n-1}\\
 &\ge \bar{d}(G) - \frac{q-2}{n-1}
\end{align*}
as desired.
\end{proof}

Suppose that $v\in V(G)$ and that we split off several pairs of edges incident with
$v$ and then delete the vertex $v$. Then we say that
we have \DEF{split} the vertex $v$. Let us observe that if the graph obtained from $G$ after splitting a vertex contains a strong immersion of a graph $H$, then so does $G$.

In the following lemma, $\log(\cdot)$ denotes the natural logarithm.

\begin{lemma}
\label{density}
Let\/ $G$ be a simple $n$-vertex graph with $m$ edges.
Then $G$ contains a strong immersion of\/ $K_q$,
where $q = \Bigl\lfloor \frac{m}{n \log (n^2/m)} + \tfrac{1}{3} \,  \Bigr\rfloor $.
\end{lemma}

\begin{proof}
We shall repeatedly apply Lemma \ref{lem:split1} to split vertices until we have only
$\lceil m/n \rceil + 1$ vertices remaining.  If we encounter a graph in this process which contains a $K_q$ subgraph, then we are finished.  Otherwise, the average degree of the resulting graph $H$ satisfies
\begin{align*}
  \bar{d}(H)
    &\ge \frac{2m}{n} - \frac{q-2}{n-1} - \frac{q-2}{n-2} - \cdots - \frac{q-2}{\lceil m/n \rceil + 1}\\
    &> \frac{2m}{n} - (q-2)\int_{\lceil m/n \rceil}^{n-1} \frac{1}{x}\,dx\\[1mm]
    &\ge \frac{2m}{n} - (q-2) \log (n^2/m) \\
    &\ge \frac{m}{n} + \frac{5}{3} \log(n^2/m)
    \ge \frac{m}{n} + 1
\end{align*}
(here the last inequality follows from $n^2 \ge 2m$ since $G$ is simple).  But this is a contradiction since the (simple) graph $H$ has only $\lceil m/n\rceil + 1$ vertices.
\end{proof}

The {\it edge-cut} $\delta(X)$ for a vertex subset $X$ of a graph is the set of edges with exactly one vertex in $X$.

\begin{lemma}
Every graph of minimum degree $2k$ contains a subgraph $H$ so that $H$ has $k$ edge-disjoint spanning trees.
\end{lemma}

\begin{proof}
If there exists an edge-cut of size less than $2k$ in $G$,
then choose such an edge-cut $\delta(X)$ so that $X$ is
minimal.  Otherwise, we take $X = V(G)$, and in either case we set $H$ to be the subgraph induced by $X$.  Suppose (for a contradiction) that
$H$ does not have $k$ edge-disjoint spanning trees.  Then it follows from a theorem of Nash-Williams \cite{NW} and Tutte \cite{Tu} that there exists a partition of $X$ into $t \geq 2$ blocks, say $\{X_1,\ldots,X_t\}$, so that the number of edges with ends in
distinct blocks of this partition is less than $k(t-1)$.  But then we find that
\begin{equation*}
\sum_{i=1}^t | \delta(X_i) |
	< 2k(t-1) + |\delta(X)|
	\le 2kt,
\end{equation*}
so there exists $1 \le i \le t$ so that $|\delta(X_i)| < 2k$.  This contradicts our choice of $X$, so we conclude that $H$ contains
$k$ edge-disjoint spanning trees.
\end{proof}

\begin{lemma}
\label{eul-sg}
Let $G$ be a simple graph of minimum degree $4k$. Then $G$ contains an Eulerian subgraph with minimum degree at least\/ $2k$.
\end{lemma}

\begin{proof}
It follows from the previous lemma that we may choose a subgraph $H$ of $G$ so that $H$ contains $2k$ edge-disjoint spanning trees
$T_1, T_2, \ldots, T_{2k}$.  For every edge $e \in E(H) \setminus E(T_1)$ let $C_e$ be the fundamental cycle of $e$ with respect to $T_1$,
and let $S$ be the set of edges obtained by taking the symmetric difference of all of the cycles $C_e$ for $e \in E(H) \setminus E(T_1)$.
Define $H'$ to be the subgraph of $H$ with edge set $S$. Since a cycle is an Eulerian graph, and the symmetric difference of Eulerian graphs is an Eulerian graph, the graph $H'$ is Eulerian. Furthermore, $S$ contains all edges in
$E(H)\setminus E(T_1) \supseteq \cup_{j=2}^{2k} E(T_j)$. Since each of
the trees $T_j$ contributes at least one to the degree of each vertex,
$H'$ has minimum degree at least $2k-1$.  Since every degree is even and $2k-1$ is odd, we conclude that $H'$ is an
Eulerian graph with minimum degree at least $2k$, as required.
\end{proof}

\section{Proof of Theorem \ref{main}}

By using the auxiliary results from the previous section, we will be able to start the proof
with a simple Eulerian graph $G$ with minimum degree at least $100t$.
The main idea of the proof is to split off all edges incident with a vertex. If the complement of the neighb{ou}rhood of that
vertex has a perfect matching, the splitting can be done in such a way that the graph remains simple and has
minimum degree at least $100t$.
However, if this is not the case, any complete splitting of the vertex gives rise to some parallel edges. We show that
one can keep all parallel edges (with an exception of a single double edge) being contained in a small subset $A$ of
vertices which is completely joined to a relatively large vertex set $B$. When doing so, we either find a
$K_t$-immersion in $G$, or produce a splitting in which something improves. Our goal is to either make the set $A$
empty, or make $A$ larger and $B$ smaller and ``denser''. To control the
``density'' of $B$, we keep track of the previous splittings that give rise to large disjoint matchings in $B$, and henceforth certify that there are
many edges in $B$. Formally, this is controlled with a third parameter $s$ (the first two parameters
being $a=|A|$ and $b=|B|$), denoting how many times we have made a splitting of a vertex
in which parallel edges may have occurred.

In addition to the lemmas from the previous section, our proof requires the following consequence of the Edmonds-Gallai structure theorem for matchings, see \cite{LP}.  We say that a graph is \emph{hypomatchable} if removing any vertex results in a graph with a perfect matching, and we let
${\mathit comp}(G)$ $({\mathit oddcomp}(G))$ denote the number of components (components of odd order) of the graph $G$.

\begin{theorem}[Edmonds-Gallai]
\label{ed-gal}
For every graph\/ $G$ without a perfect matching, there exists $X \subseteq V$ so that every component of\/ $G-X$ of odd order is
hypomatchable and so that ${\mathit oddcomp}(G-X) > |X|$.
\end{theorem}

The following lemma, which will be used to derive our main theorem, formalizes the outline of
the proof as explained at the beginning of this section.

\begin{lemma}
Let\/ $G$ be an Eulerian graph with minimum degree $\ge 100t$, let $s \ge 0$ be an integer, and let $A,B \subseteq V(G)$. Let\/ $a = |A|$ and $b = |B|$ and assume that the following properties hold:
\begin{enumerate}
\item[\rm (i)] $A \cap B = \emptyset$ and every vertex in $A$ is adjacent to every vertex in $B$.
\item[\rm (ii)] Every loop has its vertex in $A$, and all but at most one parallel class of edges have both ends in $A$.
\item[\rm (iii)] If there is a parallel class of edges without both ends in $A$, then it has size two, has one end in $A$, and the other end has at most $50t$ neighb{ou}rs in~$B$.
\item[\rm (iv)] There exist $s$ edge-disjoint matchings in $B$ each with size $\ge b - 53t$.
\item[\rm (v)] $2a + b + 2s \ge 100t$.
\item[\rm (vi)] $a > 0$ and\/ $72t \le b \le 100t$.
\end{enumerate}
Then $G$ contains a strong immersion of $K_t$.
\end{lemma}

\begin{proof}
Suppose for a contradiction that the lemma does not hold and consider a counterexample so that $|V(G)|$ is minimum,
and subject to this $b$ is minimum.  Note that the lemma is trivial for $t = 1$ so we may further assume that $t \ge 2$.  If there exists a parallel class without both ends in $A$ then let $u$ be it's end in $A$ and let $u'$ be its other end.
Otherwise, we let $u$ be an arbitrary vertex in $A$.  Let $H$ be the complement of the graph induced by $N(u) \setminus A$.  Now, if we have the parallel
class $uu'$ then $u' \in V(H)$ and we modify $H$ by adding a clone $u''$ of the vertex $u'$ which is not adjacent to the original.  Our proof involves a sequence of claims.

\bigskip

\noindent{(1)} $H$ does not have a perfect matching.

\smallskip

Suppose for a contradiction that (1) fails and choose a perfect matching $M$ of $H$ so that $M$ has the maximum number of edges with both ends in the set $B$.  Suppose for a contradiction that there is a subset $M' = \{ v_1w_1, v_2w_2, \ldots, v_{6t-1}w_{6t-1} \}$
of $M$ so that no edge in $M'$ is incident with the clone (if it exists) and so that $v_i \in B$ and $w_i \not\in B$ for $1 \le i \le 6t-1$.
If there exist $1 \le i < j \le 6t-1$ so that $v_i v_j, w_i w_j \in E(H)$, then we could exchange the edges $v_i v_j, w_i w_j$ for $v_i w_i, v_j w_j$ and thus improve $M$.  It follows that at least one of these edges
is not in $H$.  But then, in our original graph $G$, either the subgraph induced by $\{ v_1, v_2, \ldots, v_{6t-1} \}$ or the subgraph induced by $\{ w_1, \ldots, w_{6t-1} \}$ has at least $\tfrac{1}{2}{6t-1 \choose 2}$
edges.  Then applying Lemma \ref{density} to this subgraph gives us an immersion of $K_t$ which is a contradiction (to check this,
it suffices that $\frac{m}{n \log{n^2/m}} + \frac{1}{3} \ge t$ when $n = 6t-1$ and $m= \frac{1}{4}(6t-1)(6t-2)$ but this is equivalent to
$\log(4 + \frac{4}{6t-2}) \le \frac{3}{2}$ which holds for all $t \ge 2$).  Thus, we may assume no such subset of $M$ exists.

Now we shall view the perfect matching $M$ of $H$ as a set of nonedges in $G$ (we treat an edge in $M$ with endpoint $u''$ as a nonedge in $G$ with the corresponding endpoint $u'$), and we split the vertex $u$ so that the edges between $u$ and $V \setminus A$ split off to form $M$, and we split off the remaining edges between $u$ and $A \setminus \{u\}$ arbitrarily (this is possible since $u$ has even degree).
Let us call this newly formed graph $G'$.  If $A = \{u\}$, then $G'$ is a simple graph and we
get a smaller counterexample by choosing a vertex $w \in V(G')$ and taking $G'$ together with the sets $A' = \{w\}$ and $B'$ a subset of
$100t$ neighb{ou}rs of $w$, and $s' = 0$.  Thus $A \neq \{u\}$ and we now set $A' = A \setminus \{u\}$ and $B' = B$ and $s' = s+1$.
We claim that $G', A', B', s'$ yield a smaller counterexample.  All properties except (iv) are immediate, so it remains only to check this one.  Let $M^*$ be the subset of edges of $M$ which are not incident with the clone in $H$
and have both ends in $B$.  Now we have that $M^*$ is a matching of nonedges in $G$ which covers all vertices of $B$ except for a set of at most $6t-2$ vertices that are joined by $M$ to vertices outside $B$, and except for
possibly one vertex which was joined by $M$ to the clone.  It follows that $M^*$ has size at least $\frac{1}{2}(b - (6t-2) - 1) > \frac{1}{2}b - 3t \ge b - 53t$ (since $b \le 100t$).  It follows that property (iv) holds as desired.

\bigskip

\noindent{(2)} $H$ is not hypomatchable.

\smallskip

Suppose for a contradiction that (2) fails. If every vertex in $B$ except for possibly $u'$ more than $50t$ neighb{ou}rs in $B$, then the subgraph induced by $B$ has more than $\frac{1}{2} (b-1)50t \ge 1800t^2 - 25t \ge 1700t^2$ edges and at most $100t$ vertices
so applying Lemma \ref{density} to it gives us an immersion of $K_t$ which is contradictory.  Thus, we may choose a vertex $w \in B \setminus \{u'\}$ so that $w$ has at most $50t$ neighb{ou}rs in $B$.
Since the graph $H$ is hypomatchable, the graph $H - w$ has a perfect matching, and we choose one such perfect matching $M$ so that $M$ has the maximum number of edges with both ends in the set $B$.  As in the previous case,
we find that there does not exist a subset $M'$ of $M$ consisting of $6t-1$ edges not incident with the clone which join vertices in $B$ to vertices outside $B$.

As in the previous case, we now view $M$ as a set of nonedges in $G$ and we shall split the edges between $u$ and $V(G) \setminus (A \cup \{w\})$ to form $M$, and we split the remaining edges between $u$ and
$\{w\} \cup (A \setminus \{u\})$ arbitrarily.  As before, let us call this newly formed graph $G'$.  Note that because $u$ has even degree, this operation is possible, and furthermore, the edge between $u$ and $w$ must
be split off with another edge between $u$ and a vertex in $A \setminus \{u\}$.  In particular, this implies that the set $A' = A \setminus \{u\}$ is still nonempty.  Now we set $B' = B$ and $s' = s+1$ and claim that $G'$
together with $A'$, $B'$, and $s'$ form a smaller counterexample.  This time our splitting operation has created a new parallel edge between $A'$ and $B'$ (namely the new edge incident with $w$), but we also
eliminated any such existing parallel edge, so (ii) and (iii) still hold.  All other properties except (iv) are immediate, and the proof of (iv) is similar to the previous case.  As before if we take $M^*$ to be the subset
of edges in $M$ which are not incident with the clone in $H$ and have both ends in $B$, then $M^*$ is a matching of nonedges in $G$ which covers all of $B$ except for a set of at most $6t-2$ vertices which were
matched by $M$ to vertices outside $B$, and except for possibly one vertex joined by $M$ to the clone, and except for the vertex $w$.  It follows that $M^*$ has size at least $\frac{1}{2}(b - 6t) \ge b - 53t$ as desired.

\bigskip

\noindent{(3)} If $B' \subseteq B$, there exist $s$ edge-disjoint matchings in $B'$ each of size at least $|B'| - 53t$.

\smallskip

This is clear, since the removal of $k$ vertices from $B$ can only decrease the size of a matching by $k$ edges.

\bigskip

\noindent{(4)} $b \ge 74t$.

\smallskip

If $a \ge t$, then it follows from Lemma \ref{bip-im} that $G$ immerses $K_t$ and we are finished.  Similarly,
if $s \ge 12t$ then consider a subgraph $G'$ of $G$ induced by $72t$ vertices from $B$.  It follows from (3) that $G'$ has at
least $228t^2$ edges and now applying Lemma \ref{density} to $G'$ we
find that $G'$ immerses $K_t$, a contradiction.  Now (v) gives us
$b \ge 100t - 2a - 2s > 100t - 2t - 24t = 74t$, as desired.

\bigskip

Our next step will be to apply Theorem \ref{ed-gal} to $H$. It follows from (1) that this applies nontrivially to give us a set $X$ as in the theorem. Note that $|X| < |H|/2$, since $oddcomp(H-X)>|X|$, so $|H-X|> |H|/2 \geq b/2\geq 37t$. We claim that there is a single component $K$ of $H-X$ which contains all but at most $t$ vertices of $H-X$. To see this, consider a bipartition of the components $H-X$ which is as balanced as possible (in terms of number of vertices on each side). If each side has at least $t+1$ vertices then, after possibly removing our cloned vertex, we still have a partition where both sides have $\ge t$ vertices.  Back in the original graph this gives us a $K_{t,t}$, and now applying Lemma \ref{bip-im} gives a contradiction. So one side of the bipartition must have at most $t$ vertices, and the other set must have more than $2t$ (in fact, more than $36t$). Since the bipartition is as balanced as possible, the larger set has only one component so we have our claim. Now, since $K$ contains all but at most $t$ vertices of $H-X$, we know that $t+1\geq comp(H-X)>|X|$. Hence $K$ contains all but at most
$2t$ vertices of $H$.

It follows from (1) and (2) that $oddcomp(H-X)\geq 2 $. Let $Y$ be the union of the vertex sets of the components of $H-X$ other than $K$.  In the original graph $G$, every
vertex in $Y$ is adjacent to every vertex in $V(K)$ so every vertex in $Y$ is adjacent to all but at most $2t$ vertices in $B$.  It follows from
this and (iii) that if we have a cloned vertex, then this clone is not in $Y$.  If there is a clone and it appears in $X$, then let $X'$ be the
subset of $X$ obtained by removing this clone; otherwise, let $X' = X$.  We now have the following properties:
\begin{itemize}
\item $|X'| + |Y| \le 2t$.
\item $|Y| \ge |X'|$.
\item In $G$, every vertex in $Y$ is adjacent to every vertex in $B \setminus (X' \cup Y)$.
\end{itemize}
Set $A' = A \cup Y$, set $B' = B \setminus (X' \cup Y)$ and set $s' = s$.  It follows from the above properties (together with (3) and (4)) that $G$ with
$A'$, $B'$, and $s'$ form a smaller counterexample, thus completing the proof.
\end{proof}

\bigskip

\begin{proof}[Proof of Theorem \ref{main}]
Let $G$ be a simple graph with minimum degree at least $200t$.
It follows from Lemma \ref{eul-sg} that $G$ has an Eulerian subgraph $G'$ of minimum degree at least $100t$.  Now choose a vertex $u \in V(G')$ and apply the previous lemma to $G'$ with the sets $A = \{u\}$, $B$ a set of $100t$ neighb{ou}rs of $u$, and $s =0$.  This yields
a strong immersion of $K_t$ in $G$ as desired.
\end{proof}

\section{Clique immersions in dense graphs}

\label{sect:dense}

In this section we treat the dense case, when graphs have quadratically many edges. Let us recall that a $1$-immersion of a graph $H$ into a graph $G$ is an immersion of $H$ into $G$ such that each of the paths of the immersion has exactly one internal vertex. We will prove Theorem \ref{main1immerse} which shows that dense graphs contain 1-immersions of large cliques. More precisely, a graph on $n$ vertices with $2cn^2$ edges contains a strong $1$-immersion of the complete graph on at least $c^2 n$ vertices. Alon, Krivelevich, and Sudakov \cite{AKS} proved that every graph $G$ on $n$ vertices with $cn^2$ edges contains a $1$-subdivision of a complete graph with at least $c'\sqrt{n}$ vertices for some positive $c'$ depending on $c$. In comparison, our Theorem \ref{main1immerse} shows that, over all graphs of a given order and size, the order of the largest guaranteed clique $1$-immersion is about the square of the order of the largest guaranteed clique $1$-subdivision.


In a graph, the \DEF{common neighb{ou}rhood} of a set $S$ of vertices is the set of all vertices adjacent to all vertices in $S$. The following is the key lemma for the proof of Theorem \ref{main1immerse}. It uses a powerful probabilistic technique known as {\em dependent random choice} (see the survey \cite{FoSu} for more details). The basic idea is that, in a dense graph, the common neighb{ou}rhood $X$ of a small random subset of vertices likely is both large and has only few small subsets $S \subset X$ which have small common neighb{ou}rhood.
The {\it codegree} $d(u,v)$ of a pair of vertices $u,v$ in a graph is the number of vertices adjacent to both $u$ and $v$.

\begin{lemma}
\label{mainlem}
Let\/ $H=(A,B;E)$ be a simple bipartite graph with $n$ vertices and at least $cn^2$ edges, where $|B| \geq |A|$ and $c^2n > 2$. Let $s=\lceil c^2n \rceil$.
For a pair $u,v$ of distinct vertices, let $w(u,v)=1/d(u,v)$ if $d(u,v)<2s$ and $w(u,v)=0$ otherwise.
Define the weight of a vertex $v$ with respect to a set $S$ to be $w_S(v):=\sum_{u \in S \setminus \{v\}}w(u,v)$. Then there is a subset $U \subset A$ such that
$|U| = s$ and all vertices $v \in U$ satisfy $w_U(v) < \tfrac{1}{2}$.
\end{lemma}

\begin{proof}
Define $\alpha=|B|/n$. Pick two vertices from $B$ at random with repetitions. Let $X$ denote the set of common neighb{ou}rs of the selected random pair. The probability that a given vertex
$v \in A$ is in $X$ is $({\mathit deg}(v)/|B|)^2$. Hence,
$$\mathbb{E}[|X|]=\sum_{v \in A}\left(d(v)/|B|\right)^2 \geq |A|\left(\frac{cn^2}{|A||B|}\right)^2 = \frac{c^2n}{\alpha^2(1-\alpha)}
\geq \frac{27}{4}c^2n,$$
where the first inequality uses convexity of the function $f(x)=x^2$ together with Jensen's inequality, and the second inequality
uses that the maximum of $\alpha^2(1-\alpha)$ with $\alpha \in [0,1]$ is $4/27$ which occurs at $\alpha=2/3$.

Let $Y=\sum_{u,v \in X,u \not = v}w(u,v)=\sum_{v \in X}w_X(v)$. For a pair $u,v \in A$ of distinct vertices, the probability that $u$ and $v$ are in $X$ is $\left(d(u,v)/|B|\right)^2$. Thus the expected value of $Y$ satisfies
\begin{align*}
  \mathbb{E}[Y]
  &= \sum_{u,v \in A,u \not = v,d(u,v)<2s}\frac{1}{d(u,v)}\left(d(u,v)/|B|\right)^2 \\[1mm]
  &= |B|^{-2}\sum_{u,v \in A, u \not = v,d(u,v)<2s}d(u,v) \\[1mm]
  &\le |B|^{-2}|A|^2 (2s-1).
\end{align*}
Hence there is a choice of one or two vertices from $B$ such that the set $X$ satisfies
\begin{align*}
  |X|-2Y &\geq \mathbb{E}[|X|-2Y]=\mathbb{E}[|X|]-2\mathbb{E}[Y] \\
         &\ge \frac{27}{4}c^2n-2|B|^{-2}|A|^2(2s-1) \geq \frac{27}{4}c^2n-4s+2 > s.
\end{align*}
We used linearity of expectation, $|B| \geq |A|$, $c^2n > s-1$ and $s\ge3$. Notice that $2Y=2\sum_{v \in X}w_X(v)$ is an upper bound on the number of vertices $v \in X$ with $w_X(v) \geq \tfrac{1}{2}$. Delete from $X$ each vertex $v$ with $w_X(v) \geq \tfrac{1}{2}$, and let $U'$ denote the resulting subset. Since $|X|-2Y > s$, and $2Y$ is an upper bound on the number of vertices deleted from $X$ to obtain $U'$, we have $|U'| >s$. Let $U$ be any subset of $U'$ with $|U|=s$. For each $v \in U$, we have $w_U(v) \leq w_X(v) < \tfrac{1}{2}$, where the first inequality follows from the definition of the weight and $U \subset X$. This completes the proof.
\end{proof}

\begin{proof}[Proof of Theorem \ref{main1immerse}]
If $c^2n \leq 1$, then we can pick any vertex of $G$ to be the $1$-immersion of $K_1$. If $1 < c^2n \leq 2$, then $c> 1/\sqrt{n}$, and the number of edges of $G$ is at least $2cn^2 \geq 2n^{3/2} > n/2$, and hence there is a vertex of degree at least $2$ in $G$. So there is a path with two edges, and this is a $1$-immersion of $K_2$. So we may assume that $c^2n>2$.

As $G$ has $n$ vertices and at least $2cn^2$ edges, there is a bipartite subgraph $H$ that has at least $cn^2$ edges, namely the maximum cut has this property. Let $A$ and $B$ denote the bipartition of $H$, with $|B| \geq |A|$. Let $s=\lceil c^2n \rceil$. By Lemma \ref{mainlem}, there is a vertex subset $U \subset A$ with $|U| =s$ and all vertices in $U$ have weight with respect to $U$ less than $\tfrac{1}{2}$. This implies that for each $i$, $1 \leq i < 2s$, and each vertex $v \in U$, there are less than $i/2$ vertices $u \in U \setminus \{v\}$ with $d(u,v) \leq i$. Indeed, otherwise the weight of $v$ with respect $U$ would be at least $\frac{i}{2} \cdot \frac{1}{i} = \tfrac{1}{2}$, a contradiction.

We will find a $1$-immersion of the complete graph on $s$ vertices, where $U$ is the set of vertices of the clique immersion. So we need to find, for each pair $u,v$ of distinct vertices in $U$, a path $P_{uv}$ from $u$ to $v$ with one internal vertex, such that all these paths are edge-disjoint. The internal vertices of the paths will be chosen from $B$. Let us order the pairs $u,v$ of distinct vertices of $U$ by increasing value of $d(u,v)$. We will select the paths in this order. When it is time to pick the path between $u$ and $v$, we will pick any possible vertex $z \in B$ for the internal vertex. That is, if $z \in B$ is such that $uz$ and $zv$ are edges of $G$ and these edges are not in any of the previously chosen paths, then we choose $P_{uv} = uzv$. Suppose that the pair $u,v$ has $r$ common neighb{ou}rs in $B$. So far we have picked at most $|U|-1=s-1$ paths containing $u$ and at most $|U|-1=s-1$ paths containing $v$.
Thus, if $r\ge 2s$, there is a common neighb{ou}r $z$ that can be used for the path $P_{uv}$.

Suppose now that $r<2s$. Since $u$ and $v$ have weight less than $\tfrac{1}{2}$, the number of pairs $(u,x)$ with $x \in U \setminus \{u\}$ such that $d(u,x) \leq r$ is less than $r/2$, and similarly the number of pairs $(v,x)$ with $x \in U \setminus \{v\}$ such that $d(v,x) \leq r$ is less than $r/2$. Hence, less than $r/2$ edges containing $u$ are in previously selected paths, and similarly less than $r/2$ edges containing $v$ are in previously selected paths. Therefore, there is a vertex $z$ adjacent to both $u$ and $v$ such that the edges $uz$ and $zv$ are not on previous paths. We thus can choose the path $P_{uv}=uzv$. Once we are done picking the paths, we have the desired $1$-immersion. Observe that this immersion is strong.
\end{proof}

As the following proposition shows, Theorem \ref{main1immerse} is tight in the sense that the exponent in $c^2$ cannot be improved.

\begin{proposition}
For every $\varepsilon > 0$ and integer $n_0$, there exists a simple graph $G$ on $n>n_0$ vertices that contains no $1$-immersion of the complete graph
on $c^{2-\varepsilon}n$ vertices, where $c=|E(G)|/{n \choose 2}$.
\end{proposition}

\begin{proof}
Fix an integer $t>1+2(2-\varepsilon)/\varepsilon$ and let $n>n_0$ satisfy $n^{\frac{\varepsilon}{2} - \frac{2-\varepsilon}{t-1}}\ge t^{9-4\varepsilon}$ and
$n=r^{2(t-1)}$ with $r$ a multiple of $2t^4$ so that $\frac{1}{2}t^{-4}n^{\frac{1}{2}+\frac{1}{t-1}}$ is an integer.
Consider the random graph $G(n,c)$ with $c=t^{-4}n^{-\frac{1}{2}-\frac{1}{t-1}}$. By our choice of $n$ and $c$, we have $c{n \choose 2}$, the expected number of
edges of the random graph $G(n,c)$, is an integer. A $1$-immersion of $K_t$ has,
including the vertices in the paths, at most $v:=t+{t \choose 2}=\frac{t^2+t}{2}$ vertices,
and has $e:=2{t \choose 2}=t^2-t$ edges. The
number of graphs with $v$ vertices and $e$ edges is at most ${{v \choose 2} \choose e} < v^{2e}$. Therefore, the expected
number of $1$-immersions of $K_t$ in $G(n,c)$ is less than
$v^{2e}c^en^v=v^{2e}t^{-4e} \leq 1$. It is easy to check that the expected
number of copies of any particular subgraph $H$ in a uniform random graph on
$n$ vertices and edge density $c$ is at most the expected number of copies of
$H$ in $G(n,c)$. In particular, this implies that there is a graph $G$ on $n$
vertices with edge density $c$ and no $1$-immersion of $K_t$.
By the choice of $c$ and $n$, we have
$$c^{2-\varepsilon}n=t^{-8+4\varepsilon}n^{\frac{\varepsilon}{2}-\frac{2-\varepsilon}{t-1}}\ge t,$$
thus $G$ contains no $1$-immersion of the complete graph on $c^{2-\varepsilon}n$ vertices.
\end{proof}

\section{Clique immersions in very dense graphs}

\label{sect:examples}

We show in this section that immersing a clique into a very dense graph is closely related to the chromatic index of the complement of the graph. The {\it chromatic index} $\chi'(G)$ of a graph $G$ is the minimum number of col{ou}rs in a proper edge-col{ou}ring of $G$, that is, in a col{ou}ring of the edges of $G$ so that no two edges having a vertex in common receive the same col{ou}r. Letting $\Delta(G)$ denote the maximum degree of $G$, we have $\chi'(G) \geq \Delta(G)$, as the edges containing a vertex of maximum degree must be different col{ou}rs in a proper edge-col{ou}ring of $G$. Vizing's classical theorem says that this is bound is close to best possible, namely $\chi'(G) \leq \Delta(G)+1$ holds for every graph $G$.

Paul Seymour \cite{Se} (see \cite{DKMO} or \cite{LM}) found examples of graphs of minimum degree $d$ that do not contain $K_{d+1}$-immersions, for every $d\ge 10$. In this paper, we exhibit a large class of similar examples which generalize those provided by Seymour (see Theorem \ref{counterexamples} below).

\begin{lemma}\label{obvious}
Suppose $G$ has a $K_t$-immersion on a vertex subset $J$. Then $G$ contains an $K_t$-immersion on $J$ in which the edges between adjacent vertices in $J$ are used as the paths between these vertices.
\end{lemma}

\begin{proof}
Indeed, suppose $v$ and $w$ are adjacent and edge $vw$ is not used as the path $P_{vw}$ between $v$ and $w$. If the edge $vw$ is not in any of the paths in the clique immersion, we can just replace $P_{vw}$ by this edge. Otherwise, $vw$ is in a path $P$ between two vertices $u,u'$. We can then replace $P_{vw}$ by $vw$, and replace in $P$ the edge $vw$ by $P_{vw}$. This creates a walk from $u$ to $u'$ containing a path from $u$ to $u'$ that can be used for the immersion. We can do this for each edge contained in $J$ and get a $K_t$-immersion on $J$ in which the edges between adjacent vertices in $J$ are used as the paths between these vertices.
\end{proof}

As there are many examples of regular graphs for which the bound in Vizing's theorem is tight, the following theorem gives an interesting large class of examples of graphs of minimum degree $d$ that do not contain a clique immersion of $K_{d+1}$. They generalize previously found example of Seymour.

For a graph $G$ and vertex subsets $S$ and $T$, let $e(S)$ denote the number of edges with both endvertices in $S$, let $\bar e(S)={|S| \choose 2}-e(S)$ denote the number of pairs of vertices in $S$ that are nonadjacent in $G$, and let $e(S,T)$ ($\bar e(S,T)$) denote the number of pairs $(s,t) \in S\times T$ that are adjacent (respectively, nonadjacent) in $G$.

\begin{theorem}\label{counterexamples}
  Suppose $H_1,\ldots,H_t$ are simple $D$-regular graphs, each with chromatic index $D+1$, where $t > \tfrac{1}{2}D(D+1)$. Let $G$ be the complement of the graph formed by taking the disjoint union of $H_1,\ldots,H_t$. Letting $n$ denote the number of vertices of\/ $G$, the minimum degree of\/ $G$ is $n-1-D$ but $G$ does not contain an immersion of the complete graph on $n-D$ vertices.
\end{theorem}

\begin{proof}
Suppose for a contradiction that $G$ contains an immersion of the complete graph on $n-D$ vertices. Let $J$ be the set of vertices of this clique immersion, and $S=V(G) \setminus J$ be the remaining $D$ vertices. By Lemma \ref{obvious}, we may assume that the edges between adjacent vertices in $J$ are used as the paths between its vertices. Thus, for each pair $v,w$ of nonadjacent vertices in $J$, there is a path $P_{vw}$ joining $v$ and $w$, such that each edge of this path contains at least one vertex in $S$, and these paths are edge-disjoint.

As each vertex of $G$ (and, in particular, each vertex in $S$) has exactly $D$ non-neighb{ou}rs, the number $X$ of edges of $G$ containing at least one vertex in $S$ is
$$|S|(n-1-D)-e(S)=D(n-1-D)-e(S).$$ For the same reason, the number $Y$ of pairs of nonadjacent vertices in $J$ is
$$\frac{|J|D - \bar e(S,J)}{2}=\frac{(n-D)D - \bar e(S,J)}{2}.$$

Every path $P_{vw}$ between nonadjacent vertices in $J$ uses at least two edges and all of its edges contain a vertex in $S$.
Hence the number of such paths that use at least three edges is at most
$$X-2Y=-D-e(S)+\bar e(S,J) = -D + D^2 - e(S) - 2\bar e(S) \leq \frac{D^2-D}{2},$$
where we used the fact that $\bar e(S,J) = D^2 - 2 \bar e(S)$, which follows from every vertex in $S$ having $D$ non-neighb{ou}rs, together with $e(S) + 2\bar e(S) \geq e(S) + \bar e(S) = \binom{D}{2}$.

As $t>\tfrac{1}{2}D(D+1)$, there is at least one of the $H_j$, all of whose vertices belong to $J$, and for all pairs of vertices $v,w$ that are adjacent in $H_j$ (these pairs are nonadjacent in $G$), the path $P_{vw}$ has exactly two edges, $vs$, $sw$, where $s\in S$. Consider the col{ou}ring of the edges of $H_j$ with col{ou}r set $S$, where the col{ou}r of the edge $vw$ is the internal vertex $s$ of the path $P_{vw}$. No two edges $vw$, $vw'$ of $H_j$ which share a vertex receive the same col{ou}r $s$, as otherwise both the paths $P_{vw}$ and $P_{vw'}$ would contain the edge $vs$, which contradicts the fact that these paths are edge-disjoint. Hence, this is a proper edge-col{ou}ring of $H_j$, and since $|S|=D$, this contradicts the chromatic index of $H_j$ is $D+1$.
\end{proof}

While, in very dense graphs, Theorem \ref{counterexamples} shows that the minimum degree $d$ does not yield $K_{d+1}$-immersions, the next theorem shows that the correct bound is not too far.

\begin{theorem}
If a simple graph $G$ has $n$ vertices and minimum degree $\delta \geq n-n^{1/5}$, then $G$ contains an immersion of the complete graph on $\delta-1$ vertices.
\end{theorem}

\begin{proof}
The complement $\bar G$ of $G$ has maximum degree $D=n-\delta-1 \leq n^{1/5}-1$. Let $S=\{s_0,\ldots,s_{D+1}\}$ be a set of $D+2$ vertices in $G$, each pair having distance in $\bar G$ at least $5$. Such a set $S$ exists. Indeed, if otherwise, the largest set $T$ of vertices with each pair of vertices in $T$ having distance in $\bar G$ at least $5$ satisfies $|T| \leq D+1$. The number of vertices having distance in $\bar G$ at most $4$ from at least one vertex in $T$ is at most
$$|T|\left(1+D+D(D-1)+D(D-1)^2+D(D-1)^3\right)\leq D^5-D^4+2D^2+D+1<n.$$
Hence, there is a vertex at distance in $\bar G$ at least $5$ from $T$, a contradiction.
This proves that $S$ exists.

Let $J=V(G) \setminus S$, so $|J|=n-(D+2)=\delta-1$. We will show that $G$
contains a clique immersion on $J$, which would complete the proof. For every
pair $v,w$ of vertices in $J$ that are adjacent in $G$, the clique immersion
will use the edge $vw$ for the path from $v$ to $w$. For each pair $v,w$ of
nonadjacent vertices, the path $P_{vw}$ between $v$ and $w$ will be of length
$2$ with the internal vertex of the path being in $S$, which we next specify.
Let $\phi:J \to \{1,\ldots,D+1\}$ be a proper edge-col{ou}ring of the induced
subgraph of $\bar G$ with vertex set $J$. Such a col{ou}ring exists by Vizing's
theorem since the maximum degree of this induced subgraph is at most $D$. If
$\phi(v,w)=i$, and the pairs $v,s_i$ and $s_i,w$ are edges of $G$, then we let
$P_{vw}$ be the path with these two edges. Otherwise, we let $P_{vw}$ be the
path with edges $vs_0$ and $s_0w$. Notice that the chosen paths exist as
otherwise $\bar G$ has an edge from $s_0$ to $v$ or $w$, the edge $vw$, and an
edge from $v$ or $w$ to $s_i$, contradicting that $s_0$ and $s_i$ have distance
at least $5$ in $\bar G$. Furthermore, these paths are edge-disjoint. Indeed,
if there were two such paths which contained an edge $vs_i$, these two paths
$P_{vw}$ and $P_{vw'}$ would share an endpoint $v$. If $i \geq 1$, then the
pairs $v,w$ and $v,w'$ both receive the col{ou}r $i$ in col{ou}ring $\phi$,
contradicting $\phi$ being a proper edge-col{ou}ring of the induced subgraph of
$\bar G$ with vertex set $J$. If $i=0$, then letting $j=\phi(v,w)$ and
$j'=\phi(v,w')$, we have $j \not = j'$ as $\phi$ is a proper edge-col{ou}ring of
the subgraph of $\bar G$ induced by $J$. Since $P_{vw}$ and $P_{vw'}$ both have
internal vertex $s_0$, then, in $G$, $s_j$ is not adjacent to $v$ or $w$,
$v$ is not adjacent to $w$, $v$ is not adjacent to $w'$, and $v$ or $w'$ is not
adjacent to $s_{j'}$, implying that the distance from $s_j$ to $s_{j'}$ in $\bar G$
is at most $4$, a contradiction.
\end{proof}

\section{Clique minors in line graphs}
\label{sect:6}

In this section we provide a corollary of our main result to the study of large clique minors
in line graphs. Some increased interest in these was recently triggered by Paul Seymour.
Our Corollary \ref{cor}, whose proof is given at the end of this section, answers one of his questions.

The following lemma and the prime number theorem imply that for every positive integer $n$, the line graph of $K_n$ has a clique minor of order $(\frac{1}{2}-o(1))n^{3/2}$.

\begin{lemma}\label{line1}
Let $p$ be an odd prime and $n=p^2+p+1$. The line graph of $K_n$ has a clique minor of order $n(p+1)/2$.
\end{lemma}

\begin{proof}
It suffices to find $n(p+1)/2$ edge-disjoint trees in $K_n$ each pair intersecting in at least one vertex. Since there is a finite projective plane on $n$ points, interpreting the points as vertices and lines as cliques, there is a partition of the edge set of $K_{n}$ into $n$ cliques each of order $p+1$, each pair intersecting in exactly one point. Since $p+1$ is even, by the theorem of Nash-Williams \cite{NW}, we can edge-partition the complete graph $K_{p+1}$ into $(p+1)/2$ spanning trees. These trees, $(p+1)/2$ for each of the $n$ cliques, have the property that any two of them have a vertex in common. This completes the proof.
\end{proof}

By the following lemma, the complete graph $K_{d+1}$ (and more generally, the disjoint union of copies of $K_{d+1}$) shows that the claimed bound in Corollary \ref{cor} is tight up to the constant $c$.

\begin{lemma}\label{line2}
The line graph of a graph $G$ on $n$ vertices with maximum degree $d$ does not contain as a minor a clique of order larger than $d\sqrt{n}$.
\end{lemma}

\begin{proof}
Let $G_1,\ldots,G_t$ be edge-disjoint connected subgraphs of $G$, each pair intersecting in at least one vertex. Suppose without loss of generality that
$G_1$ has the fewest number of edges, and let $e=|E(G_1)|$, $v=|V(G_1)|$, so $e \geq v-1$. The number of edges containing a vertex in $V(G_1)$ is at most $dv$. Since each $G_i$ contains at least one of these edges, $t \leq dv \leq d(e+1) \leq 2de$. Since each $G_i$ has at least $e$ edges, $t \leq |E(G)|/e \leq \frac{dn}{2e}$. These two upper bounds imply $t \leq d\sqrt{n}$.
\end{proof}

\begin{proof}[Proof of Corollary \ref{cor}]
By using Theorem \ref{main}, we first get a clique immersion of order proportional to $d$.
Now, we use the fact that in the line graph of
a clique (and hence of a clique immersion) the bound of $d^{3/2}$ is tight up
to a constant factor, cf. Lemmas \ref{line1} and \ref{line2} for details.
This completes the proof.
\end{proof}

\subsubsection*{Acknowledgement:} We are greatly indebted to Paul Seymour, Maria
Chudnovsky, and Kevin Milans for helpful conversations on the subject.


\begin{thebibliography}{99}


\bibitem{al} F. N. Abu-Khzam and M. A. Langston,
Graph coloring and the immersion order,
{\it preprint.}

\bibitem{AKS}
N. Alon, M. Krivelevich, B. Sudakov,
Tur\'an numbers of bipartite graphs and related Ramsey-type questions,
{\it Combin. Probab. Comput.} {\bf 12} (2003), 477--494.


\bibitem{BoTh} B. Bollob\'as, A. Thomason,
Proof of a conjecture of Mader, Erd\H{o}s and Hajnal on topological complete subgraphs,
{\it European J. Combin.} {\bf 19} (1998), 883--887.

\bibitem{im1} H. D. Booth, R. Govindan, M. A. Langston, and S. Ramachandramurthis,
Sequential and parallel algorithms for $K_4$-immersion testing,
{\it J. Algorithms} {\bf 30} (1999), 344--378.

\bibitem{catlin} P. A. Catlin,
A bound on the chromatic number of a graph,
{\it Discrete Math.} {\bf 22} (1978), 81--83.

\bibitem{DKMO}
M. DeVos, K. Kawarabayashi, B. Mohar, and H. Okamura,
Immersing small complete graphs,
{\it Ars Math. Contemp.} {\bf 3} (2010), 139--146.

\bibitem{Er77} P. Erd\H{o}s,
Problems and results in graph theory and combinatorial analysis,
in: Graph Theory and Related Topics (Proc. Conf. Waterloo, 1977),
Academic Press, New York, 1979, pp.\ 153--163.


\bibitem{im2} M. R. Fellows and M. A. Langston,
Nonconstructive tools for proving polynomial-time decidability,
{\it J. ACM} {\bf 35} (1998), 727--738.

\bibitem{FoSu}
J. Fox and B. Sudakov, Dependent random choice,
{\it Random Structures and Algorithms} {\bf 38} (2011), 68--99.

\bibitem{hadwiger} H. Hadwiger,
\"Uber eine Klassifikation der Streckenkomplexe,
{\it Vierteljahrsschr. Naturforsch. Ges. Z\"urich} {\bf 88} (1943), 133--142.


\bibitem{KoSz}
J. Koml\'os, E. Szemer\'edi,
Topological cliques in graphs. II,
{\it Combin.\ Probab.\ Comput.} {\bf 5} (1996), 79--90.

\bibitem{Ko} A. Kostochka,
Lower bound of the Hadwiger number of graphs by their average degree,
{\it Combinatorica} {\bf 4} (1984), 307--316.

\bibitem{LM} F. Lescure, H. Meyniel,
On a problem upon configurations contained in graphs with given chromatic number,
Graph theory in memory of G. A. Dirac (Sandbjerg, 1985), 325--331,
Ann. Discrete Math.~41, North-Holland, Amsterdam, 1989.

\bibitem{LP}
L. Lov\'asz, M.D.\ Plummer,
Matching Theory,
AMS Chelsea Publishing, Providence, RI, 2009.

\bibitem{NW} C. St. J. A. Nash-Williams, Edge-disjoint spanning trees of finite graphs,
{\it J. London Math. Soc.} {\bf 36} (1961), 445--450.

\bibitem{RS20} N.~Robertson and P.~D.~Seymour,
Graph minors. XX. Wagner's conjecture,
{\it J.~Combin.\ Theory Ser.~B\/} {\bf 92} (2004), 325--357.

\bibitem{RS23} N. Robertson and P. D. Seymour,
Graph Minors XXIII, Nash-Williams' immersion conjecture,
J. Combin. Theory Ser. B 100 (2010), 181--205.

\bibitem{RST1} N. Robertson, P. D. Seymour and R. Thomas,
{Hadwiger's conjecture for $K_6$-free graphs,
{\it Combinatorica} {\bf 13} (1993), 279--361.}

\bibitem{Se}
P. Seymour, personal communication.

\bibitem{Tho} A. Thomason,
The extremal function for complete minors,
{\it J. Combin.\ Theory Ser.~B}\/ {\bf 81} (2001), 318--338.


\bibitem{Tu} W. T. Tutte, On the problem of decomposing a graph into $n$ connected factors,
{\it J. London Math. Soc.} {\bf 36} (1961), 221--230.


\end{thebibliography}
\end{document}